\documentclass[11pt]{article} 
\usepackage{amsmath,amssymb,amsthm}
\usepackage{mathtools}
\usepackage{fancybox}  
\usepackage{graphicx}
\usepackage{color}
\usepackage{colortab}
\usepackage{colortbl}

\allowdisplaybreaks

\begin{document}

\def\fl#1{\left\lfloor#1\right\rfloor}
\def\cl#1{\left\lceil#1\right\rceil}
\def\ang#1{\left\langle#1\right\rangle}
\def\stf#1#2{\left[#1\atop#2\right]} 
\def\sts#1#2{\left\{#1\atop#2\right\}}
\def\eul#1#2{\left\langle#1\atop#2\right\rangle}
\def\N{\mathbb N}
\def\Z{\mathbb Z}
\def\R{\mathbb R}
\def\C{\mathbb C}
\newcommand{\ctext}[1]{\raise0.2ex\hbox{\textcircled{\scriptsize{#1}}}}

\newtheorem{theorem}{Theorem}
\newtheorem{Prop}{Proposition}
\newtheorem{Cor}{Corollary}
\newtheorem{Lem}{Lemma}
\newtheorem{Def}{Definition}
\newtheorem{Conj}{Conjecture}

\newenvironment{Rem}{\begin{trivlist} \item[\hskip \labelsep{\it
Remark.}]\setlength{\parindent}{0pt}}{\end{trivlist}}

\title{$p$-numerical semigroups with $p$-symmetric properties, II 
}

\author{
Takao Komatsu 
\\
\small Faculty of Education\\[-0.8ex]
\small Nagasaki University\\[-0.8ex]
\small Nagasaki 852-8521 Japan\\[-0.8ex]
\small \texttt{komatsu@nagasaki-u.ac.jp} 
}

\date{
\small MR Subject Classifications: Primary 20M14; Secondary 11D07, 20M05, 05A15, 11B25   
}

\maketitle
 
\begin{abstract} 
Recently, the concept of the $p$-numerical semigroup with $p$-symmetric properties has been introduced. When $p=0$, the classical numerical semigroup with symmetric properties is recovered. 
In this paper, we further study the $p$-numerical semigroup with $p$-almost symmetric properties. We also give $p$-generalized formulas of Watanabe and Johnson, and introduce $p$-Arf numerical semigroup and study its properties. 
      
\end{abstract}

\section{Introduction}  

For the set of positive integers $A=\{a_1,a_2,\dots,a_k\}$ ($k\ge 2$), let the {\it denumerant} function $d(n)=d(n;A)$, where denote the number of representations of a non-negative integer $n$ in terms of $a_1,\dots,a_k$ with non-negative integral coefficients. To keep the problem from becoming trivial, we assume $a_i\ge 2$ ($1\le i\le k$) and $\gcd(A)=1$. We sometimes assume that $a_i\ge 3$ ($1\le i\le k$) too; otherwise, we have $d(2 n)\ge 1$ for all the even numbers $2 n$. 

The numerical semigroup is an additive submonoid of the monoid $\mathbb N_0$, which is the set of all non-negative integers. We assume that $\gcd(a_1,a_2,\dots,a_k)=1$, which is equivalent to the fact that $\mathbb N_0\backslash S$ is finite. Then such an additive submonoid is called the {\it numerical subgroup}. Each numerical semigroup $S$ is finitely generated by $a_1,a_2,\dots,a_k\in S$ ($k\ge 2$) and is denoted by 
$$
S:=\ang{a_1,a_2,\dots,a_k}=\left\{\sum_{i=1}^k a_i x_i: a_i\in\mathbb N_0\right\}\,. 
$$ 
If any element $\alpha\in\ang{a_1,a_2,\dots,a_k}$ satisfies the condition $\alpha\ne a_i$ ($1\le i\le k$), then $A:=\{a_1,a_2,\dots,a_k\}$ is called the {\it minimaly generator system} of $S$, and this form is called the {\it canonical form} of $S$.  

In \cite{KY2}, the concept of $p$-numerical semigroups is introduced by developing a generalization of the theory of numerical semigroups based on this flow of the denumerant. 
For a non-negative integer $p$, the {\it $p$-numerical subgroup} $S_p(A)$ denotes the ideal composed from all the integers whose number of representations is more than $p$ ways. Strictly speaking, for $p\ge 1$, $0\not\in S_p$. However, $S_p^\circ(A):=S_p(A)\cup\{0\}$ becomes a numerical semigroup if $\gcd(A)=1$. For $p\ge 1$, the maximal ideal of $S_p^\circ(A)$ is nothing but $S_p(A)$ itself. Since there is no problem in investigating the properties of $S_p(A)$, it is safe to call $S_p(A)$ the $p$-numerical semigroup.  
For the set of non-negative integers $\mathbb N_0$, 
the set of the {\it $p$-gaps} is defined by $G_p(A)=\mathbb N_0\backslash S_p(A)$, which is the set of any non-negative integer whose number of representations is at most $p$. Note that for $p\ge 1$, $0\in G_p(A)$.  
If $\gcd(A)=1$, the set $G_p(A)$ is finite.  Then, there exists the laregest element, which is called the {\it $p$-Frobenius number} and denoted by $g_p(A):=g(S_p(A))$. The cardinality of $G_p(A)$ is called the {\it $p$-genus} (or the {\it $p$-Sylvester number}) and denoted by $n_p(A):=n(S_p(A))$.  
The sum of the elements of $G_p(A)$ is called the {\it $p$-Sylvester sum} and denoted by $s_p(A):=s(S_p(A))$. 
When $p=0$, 
$g(S)=g(\ang{A})=g_0(A)$ is the classical Frobenius number, and $n(S)=n(\ang{A})=n_0(A)$ is the classical genus.  $S=S_0$ is the classical numerical semiroup. 
Studying the properties of the Frobenius number and related numbers is one of the central topics of the famous Diophantine problem of Frobenius, which is also known as the Coin Exchange Problem, Postage Stamp Problem or Chicken McNugget Problem.  
The concept of the genus comes from the Gorenstein curve singularities. Such a correspondence was characterized by E. Kunz (see also \cite{HK1}.   

One of the most interesting matters in the linear Diophantine problem of Frobenius is finding explicit formulas for the Frobenius number and related numbers. So is the case even when $p\ge 1$. By using a convenient formula \cite{Ko22c,Ko23d} including Bernoulli numbers and the elements of the $p$-Ap\'ery set, explicit formulas may be given for the power sum of the elements of the $p$-gap set $\sum_{n\in G_p(A)}n^\mu$, where $\mu$ is a non-negative integer. In particular, one can obtain the general closed formula for two variables, that is, for $A=\{a,b\}$, and some explicit formulas for three variables in the case of triangular \cite{Ko22a}, repunit \cite{Ko22b}, Fibonacci \cite{KY}, Jacobsthal \cite{KP} can be given even for $p\ge 1$. For $p=0$, more explicit formulas in the special cases have been given (see, e.g., \cite{RR18} and references therein).

\section{The $p$-Ap\'ery set and convenient formulas}   

There are several generalizations for Frobenius numbers and related values. Our $p$-generalization is also a very natural generalization in terms of the generalization of the formula for finding these values from the elements of the Ap\'ery set. Furthermore, it has just recently been established that finding the $p$-Frobenius number, $p$-Sylvester number, etc., by visually capturing and configuring the elements of the Ap\'ery set, is already valid in the case of triangular numbers \cite{Ko22a}, repunits \cite{Ko22b}, Fibonacci numbers \cite{KY} and Jacobsthal numbers \cite{KP}. 

For $p\ge 0$, define the $p$-Ap\'ery set of $A=\{a_1,\dots,a_k\}$ with $a_1=\min(A)$ by 
$$ 
{\rm Ap}_p(A;a_1)=\{m_0^{(p)},m_1^{(p)},\dots,m_{a_1-1}^{(p)}\}\,, 
$$  
where for $0\le j\le a_1-1$, 
$$
{\rm(1)}~m_j^{(p)}\equiv j\pmod{a_1},\quad {\rm(2)}~m_j^{(p)}\in S_p(A),\quad {\rm(3)}~m_j^{(p)}-a_1\in G_p(A)\,. 
$$ 
That is, the $p$-Ap\'ery set constitutes a complete residue system modulo $a_1$. 
In addition, the elements of the $p$-Ap\'ery set are arranged in ascending order, which is expressed as 
$$
{\rm Ap}_p(A;a_1)=\{\ell_0(p),\ell_1(p),\dots,\ell_{a_1-1}(p)\}\,, 
$$ 
where $\ell_0(p)<\ell_1(p)<\dots<\ell_{a_1-1}(p)$. The least element $\ell_0(p)$, which is called the {$p$-multiplicity} of $S_p$, shall be useful in the later sections. Note that $\ell_0(0)=0$.   

By using the elements of the $p$-Ap\'ery set, the power sum of the elements of the set of $p$-gaps  
can be expressed (\cite{Ko22c,Ko23d}):  
\begin{align*} 
s^{(\mu)}_p(A)&:=\sum_{n\in G_p(A)}n^\mu\\ 
&=\frac{1}{\mu+1}\sum_{\kappa=0}^{\mu}\binom{\mu+1}{\kappa}B_{\kappa}a_1^{\kappa-1}\sum_{i=0}^{a_1-1}\bigl(m_i^{(p)}\bigr)^{\mu+1-\kappa}\\ 
&\qquad +\frac{B_{\mu+1}}{\mu+1}(a_1^{\mu+1}-1)\,,
\end{align*} 
where $B_n$ are Bernoulli numbers defined by 
$$
\frac{x}{e^x-1}=\sum_{n=0}^\infty B_n\frac{x^n}{n!} 
$$ 
and  $\mu$ is a non-negative integer. 
And another convenient formula is about the weighted power sum (\cite{KZ0,KZ}) 
$$
s_{\lambda,p}^{(\mu)}(A):=\sum_{n\in \mathbb N_0\backslash S_p(A)}\lambda^n n^\mu 
$$  
by using Eulerian numbers $\eul{n}{m}$ appearing in the generating function 
$$ 
\sum_{k=0}^\infty k^n x^k=\frac{1}{(1-x)^{n+1}}\sum_{m=0}^{n-1}\eul{n}{m}x^{m+1}\quad(n\ge 1)
$$ 
with $0^0=1$ and $\eul{0}{0}=1$.

When $\mu=0,1$ in the above expression, together with $g_p(A)$ we have formulas for the $p$-Frobenius number, the $p$-Sylvester number and the $p$-Sylvester sum. 
\begin{Lem}  
Let $k$ and $p$ be integers with $k\ge 2$ and $p\ge 0$.  
Assume that $\gcd(A)=1$.  We have 
\begin{align}  
g_p(A)&=\max_{0\le j\le a_1-1}m_j^{(p)}-a_1
\label{mp-g}\,,
\\  
n_p(A)&=\frac{1}{a_1}\sum_{j=0}^{a_1-1}m_j^{(p)}-\frac{a_1-1}{2}\,. 
\label{mp-n}
\\
s_p(A)&=\frac{1}{2 a_1}\sum_{j=0}^{a_1-1}\bigl(m_j^{(p)}\bigr)^2-\frac{1}{2}\sum_{j=0}^{a_1-1}m_j^{(p)}+\frac{a_1^2-1}{12}\,.
\label{mp-s}
\end{align} 
\label{cor-mp}
\end{Lem} 
   
\noindent 
{\it Remark.}  
When $p=0$, (\ref{mp-g}), (\ref{mp-n}) and (\ref{mp-s}) are the formulas by Brauer and Shockley \cite{bs62}, Selmer \cite{se77} and Tripathi \cite{pu18,tr08}, respectively: 
\begin{align*}   
g(A)&=\left(\max_{1\le j\le a_1-1}m_j\right)-a_1\,, 
\\   
n(A)&=\frac{1}{a_1}\sum_{j=0}^{a_1-1}m_j-\frac{a_1-1}{2}\,,  
\\   
s(A)&=\frac{1}{2 a_1}\sum_{j=0}^{a_1-1}\bigl(m_j\bigr)^2-\frac{1}{2}\sum_{j=0}^{a_1-1}m_j+\frac{a_1^2-1}{12}\,.  
\end{align*} 
   
It is not easy to find any explicit form of $g_p(A)$, $n_p(A)$, $s_p(A)$ and so on. However, when $k=2$, explicit closed formulas are obtained easily. When $k\ge 3$, there is no explicit formula, in general. Nevertheless, if we can find any exact structure of $m_j^{(p)}$ (though it is also enough hard, in general), we can obtain an explicit formula for such special sequences $(a_1,a_2,\dots,a_k)$.

\section{Fundamental lemmas}    

In this section, we recall some fundamental properties of the $p$-numerical semigroup \cite{KY2}.  
For a non-negative integer $p$, the $p$-numerical semigroup $S_p$, which is $p$-generated from $A$, is called {\it $p$-symmetric} if for all $x\in\mathbb Z\backslash S_p$, $\ell_0(p)+g_p(A)-x\in S_p$, where $\ell_0(p)$ is the least element of $S_p$, that is the {\it $p$-multiplicity} of $S_p$ if $p\ge 1$; $\ell_0(p)=0$ if $p=0$. 
When $p=0$, "$0$-symmetric" is just "symmetric".  
If a $p$-symmetric numerical semigroup $S_p$ further satisfies $\ell_0(p)=g_p(A)+1:=c_p(A)$, which is called $p$-conductor, then $S_p$ is called {\it $p$-completely-symmetric}.    

\begin{Lem} 
For a $p$-semigroup $S_p$ ($p\ge 0$), the following conditions are equivalent.  
\begin{enumerate}
\item[{\rm (i)}] $S_p$ is $p$-symmetric. 
\item[{\rm (ii)}] $\displaystyle \# S_p\cap\{\ell_0(p),\dots,g_p(A)\}=\# G_p\cap\{\ell_0(p),\dots,g_p(A)\}=\frac{g_p(A)-\ell_0(p)+1}{2}$. 
\item[{\rm (iii)}] If $x+y=\ell_0(p)+g_p(A)$, then exactly one of non-negative integers $x$ and $y$ belongs to $S_p$ and another to $G_p$. 
\end{enumerate}
\label{prp:p-sym}
\end{Lem} 

\begin{Lem}
For a non-negative integer $p$, $S_p$, which is $p$-generated from $A$, is $p$-symmetric if and only if $\ell_i(p)+\ell_{a-i-1}(p)=g_p(A)+\ell_0(p)+a$ ($i=1,2,\dots,\fl{a/2}$). 
\label{lem:sym-apery}
\end{Lem}

\begin{Lem}
For a non-negative integer $p$, $S_p$, which is $p$-generated from $A$, is $p$-symmetric if and only if $m_{(g+\ell+1)/2+j}(p)+m_{(g+\ell-1)/2+j}(p)=g_p+\ell+a$ ($j\in\mathbb Z$). 
\label{lem:sym-apery-m}
\end{Lem}

\begin{Lem}  
For a non-negative integer $p$, $S_p$, which is $p$-generated from $A$, is $p$-symmetric if and only if 
$$
n_p(A)=\frac{g_p(A)+\ell_0(p)+1}{2}\,. 
$$ 
\label{lem:sym-n-g}
\end{Lem} 


For a non-negative integer $p$, let $S_p(A)$ be a $p$-numerical semigroup. $x\in\mathbb Z$ is called a {\it $p$-pseudo-Frobenius number} if $x\not\in S_p(A)$ and $x+s-\ell_0(p)\in S_p(A)$ for all $s\in S_p(A)\backslash\{\ell_0(p)\}$, where $\ell_0(p)$ is the least element of $S_p(A)$, so is of ${\rm Ap}_p(A;a)$ with $a=\min(A)$. 
The set of $p$-pseudo-Frobenius numbers is denoted by ${\rm PF}_p(A)$. The {\it $p$-type} is denoted by $t_p(A):=\#\bigl({\rm PF}_p(A)\bigr)$. Notice that the $p$-Frobenius number is given by $g_p(A)=\max\bigl({\rm PF}_p(A)\bigr)$. 

For $p\ge 0$, the $p$-numerical semigroup $S_p$, which is $p$-generated from $A$, is called {\it $p$-pseudo-symmetric} if for all $x\in\mathbb Z\backslash S_p$ with $x\ne\bigl(\ell_0(p)+g_p(A)\bigr)/2\in\mathbb Z$, $\ell_0(p)+g_p(A)-x\in S_p$
. 

\begin{Lem} 
For a non-negative integer $p$, the following conditions are equivalent:  
\begin{enumerate} 
\item[{\rm (i)}] $S_p$, which is $p$-generated from $A$, is $p$-pseudo-symmetric. 
\item[{\rm (ii)}] 
$$ 
m_{(g+\ell)/2+j}^{(p)}+m_{(g+\ell)/2-j}^{(p)}=g+\ell+
\begin{cases}
2 a&\text{if $j=0$ and $(g+\ell)/2\in G_p(A)$};\\ 
0&\text{if $j=0$ and $(g+\ell)/2\in S_p(A)$};\\ 
a&\text{if $j>0$}. 
\end{cases}
$$ 
\item[{\rm (iii)}] $\displaystyle n_{p}(A)=\dfrac{g+\ell}{2}+\begin{cases}
1&\text{if $(g+\ell)/2\in G_p(A)$};\\ 
0&\text{if $(g+\ell)/2\in S_p(A)$}. 
\end{cases}$
\end{enumerate}
\label{th:ps-sym-ap}
\end{Lem} 

\begin{Lem}  
If $S_p(A)$ is $p$-symmetric, then 
\begin{enumerate} 
\item[\rm (i)] ${\rm PF}_p(A)=\{g_p(A)\}$ with $g_p(A)\not\equiv \ell_0(p)\pmod 2$. 
\item[\rm (ii)] $t_p(A)=1$ with $g_p(A)\not\equiv \ell_0(p)\pmod 2$. 
\end{enumerate} 
\label{lem:p-sym} 
\end{Lem}

\begin{Lem}  
Let $S_p(A)$ is $p$-pseudo-symmetric, then   
\begin{enumerate} 
\item[\rm (i)] $\displaystyle {\rm PF}_p(A)=
\begin{cases}
\{g_p(A),\bigl(g_p(A)+\ell_0(p)\bigr)/2\}&\text{if $\bigl(g_p(A)+\ell_0(p)\bigr)/2\in G_p(A)$};\\
\{g_p(A)\}&\text{if $\bigl(g_p(A)+\ell_0(p)\bigr)/2\in S_p(A)$}. 
\end{cases}$  
\item[\rm (ii)] $\displaystyle t_p(A)=
\begin{cases}
2&\text{if $\bigl(g_p(A)+\ell_0(p)\bigr)/2\in G_p(A)$};\\
1&\text{if $\bigl(g_p(A)+\ell_0(p)\bigr)/2\in S_p(A)$}. 
\end{cases}$ 
\end{enumerate} 
\label{lem:p-pseudo-sym} 
\end{Lem}

\begin{Lem}  
Assume that $S_p(A)$ is minimally generated by $A:=\{a_1,\dots,a_k\}$. Set $d=\gcd(a_2,\dots,a_{k})$ and $T_p(A)=\{n\in\mathbb N_0|d(n;A_d)>p\}$, where $A_d=\{a_1,a_2/d,\dots,a_k/d\}$, Then we have ${\rm Ap}(S_p,a_1)=d{\rm Ap}(T_p,a_1)$.  
\label{lem:gcd-ape}
\end{Lem} 

\begin{Lem}  
We have 
\begin{enumerate}
\item[\rm (i)] $g_p(A)=d g_p(A_d)+(d-1)a_1$. 
\item[\rm (ii)] $\displaystyle n_p(A)=d n_p(A_d)+\dfrac{(d-1)(a_1-1)}{2}$. 
\item[\rm (iii)] $\displaystyle s_p(A)=d^2 s_p(A_d)+\dfrac{a_1 d(d-1)}{2}n_p(A_d)$$+\dfrac{(a_1-1)(d-1)(2 a_1 d-a_1-d-1)}{2}$. 
\end{enumerate} 
\label{lem:gcd-ape}
\end{Lem}

\section{A $p$-generalization of Watanabe's Lemma} 

It is not so easy to find any relation between two distinct numerical semigroups. 
In this section, we give some results to keep the $p$-properties, which generalize the classically famous results.   

The first identity of the following proposition is a $p$-generalization of the result by Johnson \cite{johnson60}.  

\begin{Prop}  
Let $\{b_1,\dots,b_k\}$ be the minimal generator system of a semigroup $\ang{b_1,\dots,b_k}$, and let $\alpha$ be a positive integer such that $\alpha\in\ang{b_1,\dots,b_k}$ and $\alpha\ne b_i$ ($i=1,\dots,k$). 
Then for a positive integer $\beta$ with $\gcd(\alpha,\beta)=1$,  
\begin{align*}
g_p(\alpha,\beta b_1,\dots,\beta b_k)&=\beta g_p(\alpha,b_1,\dots,b_k)+(\beta-1)\alpha\,,\\
n_p(\alpha,\beta b_1,\dots,\beta b_k)&=\beta n_p(\alpha,b_1,\dots,b_k)+\frac{(\alpha-1)(\beta-1)}{2}\,. 
\end{align*}
\label{prp:p-johnson}
\end{Prop}  
\begin{proof}
By Lemma \ref{lem:gcd-ape}, we get the desired result.  
\end{proof}

\noindent 
{\it Remark.}  
When $p=0$, it is often possible to reduce the number of generators. For example, in \cite{tr10}, one has 
\begin{align*}
&g_0(a^k,a^k+1,a^k+a,\dots,a^k+a^{k-1})\\
&=a\cdot g_0(a^{k-1},a^k+1,a^{k-1}+1,\dots,a^{k-1}+a^{k-2})+(a-1)(a^k+1)\\
&=a\cdot g_0(a^{k-1},a^{k-1}+1,\dots,a^{k-1}+a^{k-2})+(a-1)(a^k+1)
\end{align*} 
because $a^k+1=(a-1)a^{k-1}+1\cdot(a^{k-1}+1)$. However, when $p>0$, such a reduction cannot happen in general since all number of representations must be counted.

The following is a $p$-generalization of the result by Watanabe \cite{wata73}. 

\begin{theorem} 
With the same conditions as in Proposition \ref{prp:p-johnson}, 
$S_p(\alpha,b_1,\dots,b_k)$ is $p$-symmetric if and only if $S_p(\alpha,\beta b_1,\dots,\beta b_k)$ is $p$-symmetric.    
\label{th:p-watanabe}
\end{theorem} 
\begin{proof}  
For simplicity, put $H=\{\alpha,\beta b_1,\dots,\beta b_k\}$ and $H'=\{\alpha,b_1,\dots,b_k\}$. 
First, by Lemma \ref{lem:gcd-ape}, we get  
\begin{equation}
\ell_0(p)=\beta\ell_0'(p)\,,
\label{eq:ll-ll}
\end{equation} 
where $\ell_0(p)$ and $\ell_0'(p)$ are the least elements of ${\rm Ap}_p(H)$ and ${\rm Ap}_p(H')$, respectively.   
Then, by Lemma \ref{lem:sym-n-g}, $H$ is $p$-symmetric if and only if 
\begin{align*}
&n_p(H)=\frac{g_p(H)+\ell_0(p)+1}{2}\\
&\Longleftrightarrow \beta n_p(H')+\frac{(\beta-1)(\alpha-1)}{2}=\frac{\beta g_p(H')+(\beta-1)\alpha+\ell_0(p)+1}{2}\\
&\Longleftrightarrow \beta n_p(H')=\frac{\beta g_p(H')+\ell_0(p)+\beta}{2}\\
&\Longleftrightarrow n_p(H')=\frac{g_p(H')+\ell_0'(p)+1}{2}
\end{align*} 
if and only if $H'$ is $p$-symmetric. 
\end{proof}

\noindent 
{\it Remark.}   
GAP NumericalSgps \cite{DGM,GAP} has the function {\tt DenumerantIdeal}, which calculates the number of representations (solutions) from the minimal generator system. Hence, the result yielded from $\{\alpha,\beta b_1,\dots,\beta b_k\}$ is the same as that from $(\alpha\in)\{\beta b_1,\dots,\beta b_k\}$. Hence, when $p=0$, the real situation matches the calculation, but for $p>0$, it does not because the representations by using $\alpha$ are counted in the real situation.  
For example, $25$ has $4$ representations 
$$
(0, 5, 0),\, (1, 3, 1),\, (2, 1, 2),\, (5, 1, 0) 
$$ 
in terms of $\{4,5,6\}$ and $7$ representations 
$$
(0, 0, 5, 0),\, (0, 1, 3, 1),\, (0, 2, 1, 2),\, (0, 5, 1, 0),\, (1, 0, 1, 2),\, (1, 3, 1, 0),\, (2, 1, 1, 0) 
$$ 
in terms of $\{8,4,5,6\}$.   
See Appendix for more information.

\noindent 
{\bf Examples.}\\  
$\{4,5,6\}$ is a minimal generator system of a semigroup $\ang{4,5,6}$.  
We choose $\alpha=8\in\ang{4,5,6}$ and $\beta=3$. 

We see that 
\begin{align*}
\{g_p(8,4,5,6)\}_{p=0}^{10}&=7,11,15,19,19,23,23,27,27,27,31,\\ 
\{g_p(8,12,15,18)\}_{p=0}^{10}&=37, 49, 61, 73, 73, 85, 85, 97, 97, 97, 109\,, 
\end{align*}
satisfying $g_p(8,12,15,18)=3\cdot g_p(8,4,5,6)+(3-1)\cdot 8$ for $p\ge 0$.  
The function {\tt DenumerantIdeal} yields the same sequence for $g_p(8,4,5,6)$ because it calculates the value from the minimal generator system.  

Note that 
$$
\{g_p(4,5,6)\}_{p=0}^{10}=7, 13, 19, 23, 27, 31, 33, 37, 39, 43, 43.  
$$ 
 
Concerning $H=\{8,12,15,18\}$, we have 
\begin{align*}  
S_8(8,12,15,18)&=\{72,78,80,84,86,87,88,90,92,93,94,95,96,98,\mapsto\}\,,\\
G_8(8,12,15,18)&=\{0,1,\dots,71,73,74,75,76,77,79,81,82,83,85,89,91,97\}\,. 
\end{align*} 
Since $x\in G_8(H)\cup\mathbb Z^{-}\Longleftrightarrow 
72+97-x\in S_8(H)$, we know that $S_8(H)$, which is $8$-generated from $\{8,12,15,18\}$, is $8$-symmetric. 
Concerning $H'=\{8,4,5,6\}$, we have 
\begin{align*}  
S_8(8,4,5,6)&=\{24,26,28,\mapsto\}\,,\\
G_8(8,4,5,6)&=\{0,1,\dots,23,25,27\}\,. 
\end{align*} 
Since $24+27=26+25=28+23=29+22=30+21=\cdots=51$, $S_8(H')$, which is $8$-generated from $\{8,4,5,6\}$, is $8$-symmetric.

\section{Almost symmetric}

For a numerical semigroup $S(A)$, we introduce the sets 
\begin{align*} 
H_p(A)&=\{g_p(A)+\ell_0(p)-s: s\in S_p(A)\}\,,\\ 
L_p(A)&=\{s\in\mathbb Z: s\not\in S_p(A)\,\hbox{and}\,g_p(A)+\ell_0(p)-s\not\in S_p(A)\}\,, 
\end{align*} 
satisfying $H_p(A)\cup L_p(A)\cup S_p(A)=\mathbb N_0$.  
In addition, let $K_p(A):=\{g_p(A)+\ell_0(p)-s: s\not\in S_p(A)\}$ be the $p$-canonical ideal \cite[Proposition 4]{BF97}.  

The set $L_p(A)$ implies that both corresponding $p$-symmetric elements (including the element itself is $p$-symmetric) belong to $G_p(A)$.

A numerical semigroup $S(A)$ is called {\it $p$-almost symmetric} if $L_p(A)\subset{\rm PF}_p(A)$.

\begin{Prop}
The following is equivalent.  
\begin{enumerate}
\item[\rm(i)] $S_p(A)$ is $p$-almost symmetric. 
\item[\rm(ii)] ${\rm PF}_p(A)=L_p(A)\cup\{g_p(A)\}$. 
\item[\rm(iii)] If $x\not\in S_p(A)$, then $g_p(A)+\ell_0(p)-x\in S_p(A)$ or $x\in{\rm PF}_p(A)$. 
\end{enumerate}
\label{prop:p-almostsym}
\end{Prop}
\begin{proof}
\noindent 
[(i)$\Longrightarrow$(ii)] 
If $L_p(A)\subset{\rm PF}_p(A)$, 
there exists an element $x\in{\rm PF}_p(A)\backslash L_p(A)$. 
Indeed, it is clear that $g_p(A)\in{\rm PF}_p(A)$ and $g_p(A)\not\in L_p(A)$. Assume that 
for $x'\ne g_p(A)$, $x'\in{\rm PF}_p(A)$ and $x'\not\in L_p(A)$. 
Then $y':=g_p(A)+\ell_0(p)-x'\in S_p(A)$ and $y'>\ell_0(p)$. However, $x'+s-\ell_0(p)\in S_p(A)$ does not hold for $s=y'\in S_p(A)\backslash\{\ell_0(p)\}$. Thus, there does not exist such an $x'$.    
\\\noindent 
[(ii)$\Longrightarrow$(iii)]  
Assume that $x\in G_p(A)$. If $x\in{\rm FP}_p(A)$, then it is finished. Otherwise, by $x\not\in L_p(A)$ and $x\not\in S_p(A)$ we get $g_p(A)+\ell_0(p)-x\in S_p(A)$. 
\\\noindent 
[(iii)$\Longrightarrow$(i)] 
If $x\in L_p(A)\subset G_p(A)$, then by $g_p(A)+\ell_0(p)-x\not\in S_p(A)$ we get $x\in{\rm FP}_p(A)$. 
\end{proof}

From the definitions and Lemmas \ref{lem:p-sym} and \ref{lem:p-pseudo-sym}, we see that 
\begin{align*}
S\hbox{\,is $p$-symmetric}&\Longleftrightarrow K_p(A)=S_p(A)\Longleftrightarrow H_p(A)=G_p(A)\\
&\Longrightarrow L_p(A)=\emptyset\Longleftrightarrow {\rm PF}_p(A)=\{g_p(A)\}\,,
\end{align*}
\begin{align*}
&S\hbox{\,is $p$-pseudo-symmetric}\\
&\Longrightarrow 
L_p(A)=\begin{cases}\left\{\dfrac{g_p(A)+\ell_0(p)}{2}\right\}&\text{if $\dfrac{g_p(A)+\ell_0(p)}{2}\in G_p(A)$}\\
\emptyset&\text{if $\dfrac{g_p(A)+\ell_0(p)}{2}\in S_p(A)$}
\end{cases}\\
&\Longleftrightarrow {\rm PF}_p(A)=\begin{cases}
\left\{g_p(A),\dfrac{g_p(A)+\ell_0(p)}{2}\right\}&\text{if $\dfrac{g_p(A)+\ell_0(p)}{2}\in G_p(A)$}\\
\{g_p(A)\}&\text{if $\dfrac{g_p(A)+\ell_0(p)}{2}\in S_p(A)$}\,.
\end{cases} 
\end{align*}
Therefore, both $p$-symmetric and $p$-pseudo-symmetric imply $p$-almost symmetric.  Namely, every $p$-irreducible numerical semigroup is $p$-almost symmetric.  
\bigskip

There are some typical patterns of $p$-almost symmetric numerical semigroups.  

\begin{Prop}
\begin{align*}
S_p(A)&=\{\underbrace{\ell_0(p),\ell_0(p)+1,\dots,g_p(A)-1}_{g_p(A)-\ell(p)},g_p(A)\mapsto\}\,,\\
S_p(A)&=\{\ell_0(p),g_p(A)\mapsto\}
\end{align*}
are $p$-almost symmetric.  
\end{Prop}
\begin{proof}
For the first numerical semigroup, we have 
$$
L_p(A)=\emptyset\quad\hbox{and}\quad {\rm PF}_p(A)=\{g_p(A)\}\,. 
$$ 
For the second numerical semigroup, we have 
$$
L_p(A)=\{\underbrace{\ell_0(p)+1,\ell_0(p)+2,\dots,g_p(A)-1}_{g_p(A)-\ell(p)-1}\}
$$ 
and
$$ 
{\rm PF}_p(A)=\{\underbrace{\ell_0(p)+1,\ell_0(p)+2,\dots,g_p(A)}_{g_p(A)-\ell(p)}\}\,. 
$$ 
\end{proof}
\medskip 

\noindent 
{\bf Examples.}\\
For $A=\{17,18,19\}$, we have 
\begin{align*}
S_5(A)&=\{180=\ell_0(5),\underbrace{197,198,199},\underbrace{214,\dots,218},231\mapsto\}\,,\\
G_5(A)&=\{\dots,179,\underbrace{181,\dots,196},\underbrace{200,\dots,213},\underbrace{219,\dots,229},230=g_5(A)\}\,,\\
H_5(A)&=\{\dots,179,\underbrace{192,\dots,196},\underbrace{211,212,213},230\}\,,\\
L_5(A)&=\{\underbrace{181,\dots,191},\underbrace{200,\dots,210},\underbrace{219,\dots,229}\}\,,\\
K_5(A)&=\{\underbrace{180,\dots,191},\underbrace{197,\dots,210},\underbrace{214,\dots,229},231\mapsto\}\,,\\
{\rm PF}_5(A)&=\{219,\dots,230\}\,. 
\end{align*}
We see that $H_5(A)\cup L_5(A)\cup S_5(A)=\mathbb N_0$. 
Since $L_p(A)\not\subseteq {\rm PF}_p(A)$, $S_5(A)$ is not $5$-almost symmetric.  

Similarly, for $0\le p\le 6$, $13\le p\le 20$ and $22\le p\le 28$, $S_{p}(A)$ is not $p$-almost symmetric. 
For $p=7,12,21$ and $p=29,30,44$, $S_{p}(A)$ is neither $p$-symmetric nor $p$-pseudo-symmetric but $p$-almost symmetric.  
For $8\le p\le 11$ and $31\le p\le 43$, $S_{p}(A)$ is $p$-completely symmetric, so $p$-almost symmetric. 

\medskip

Let $A=\{6,7,17\}$. Then 
\begin{align*}
S_{14}(A)&=\{126=\ell_0,131\mapsto\}\,,\\
G_{14}(A)&=\{\dots,125,127,128,129,130=g_{14}\}\,,\\
L_{14}(A)&=\{127,128,129\}\,,\\
{\rm PF}_{14}(A)&=\{127,128,129,130\}\,,\\
K_{14}&=\{126,127,128,129,131\mapsto\}\,. 
\end{align*}
Then for $z\in\{127,128,129\}$ we get $z\not\in S_{14}(A)$ and $g_{14}(A)+\ell_0(14)-z\not\in S_{14}(A)$, but $z\in{\rm PF}_{14}(A)$, satisfying the condition (iii) of Proposition \ref{prop:p-almostsym}. The condition (ii) is satisfied because $L_{14}(A)\subset L_{14}(A)\cup\{g_{14}(A)\}={\rm PF}_{14}(A)$.  

We also have  
\begin{align*}
S_{16}(A)&=\{138=\ell_0,,139,140,142\mapsto\}\,,\\
G_{16}(A)&=\{\dots,137,141=g_{16}\}\,,\\
L_{16}(A)&=\emptyset\,,\\
{\rm PF}_{16}(A)&=\{141\}\,,\\
K_{16}&=\{138,142\mapsto\}\,. 
\end{align*}
Then for all $z\in G_{16}(A)$ we get $g_{16}(A)+\ell_0(16)-z\in S_{16}(A)$, satisfying the condition (iii) of Proposition \ref{prop:p-almostsym}. The condition (ii) is satisfied because $L_{16}(A)\subset L_{16}(A)\cup\{g_{16}(A)\}={\rm PF}_{16}(A)$.  
Therefore, for $p=14,16$, $S_{p}(A)$ is neither $p$-symmetric nor $p$-pseudo-symmetric but $p$-almost symmetric.  
In fact, for $p=1,6,7,8,9,10,11,12,13,15,17,18,21,22,24,\dots$, $S_{p}(A)$ is $p$-symmetric, so $p$-almost symmetric. 
For $p=0,4,5,19,20,23,25$, $S_{p}(A)$ is $p$-pseudo-symmetric, so $p$-almost symmetric.

\section{$p$-Arf numerical semigroup} 

A numerical semigroup $S$ is called an {\it Arf numerical semigroup} if for every $x,y,z\in S$ such that $x\ge y\ge z$, then $x+y-z\in S$. Arf semigroups help to characterize Arf rings, an important class of rings in commutative algebra and algebraic geometry \cite{BDF1997,DM01,GHKR17}.  


\begin{Prop}
If $S(A)$ for $A=\{a,b\}$ with $\gcd(a,b)=1$ is an Arf numerical semigroup, then $S_p(A)$ is also an Arf numerical semigroup.  
\label{prp:arf}
\end{Prop}
\begin{proof} 
Assume that for every $x,y,x\in S_p$ such that $x\ge y\ge z$. We write $x$, $y$ and $z$ in the standard form as $x=a k_1+b h_1$, $y=a k_2+b h_2$ and $z=a k_3+b h_3$. Then by Lemma 4.3 in \cite{KY2}, $k_i\ge p b$ ($i=1,2,3$). Put $x'=x-p b=a(k_1-p b)+b h_1$, $y'=y-p b=a(k_y-p b)+b h_y$ and $z'=z-p b=a(k_3-p b)+b h_3$. Since $k_1-p b\ge 0$ and $h_i\ge 0$ ($i=1,2,3$), we get $x',y',z'\in S$ with $x'\ge y'\ge z$. As $S$ is an Arf, we have $x'+y'-z'\in S$. Hence, $x'+y'-z'$ has the standard form $x'+y'-z'=a k_0+b h_0$ with $k_0,h_0\ge 0$. Then by $x+y-z=x'+y'-z'+p a b=a(p b+k_0)+b h_0$ and Lemma 4.3 in \cite{KY2}, we have $x+y-z\in S_p$, so $S_p$ is also an Arf.  
\end{proof}

\begin{Lem}  
Let $S$ be an Arf numerical semigroup generated from $A$ with $a=\min(A)$. For a nonnegative integer $p$, let 
$p$-conductor be $c_p$, that is, $c_p=g_p(A)+1$. $\overline{c_p}$ denotes the residue modulo $a$, that is $c_p\equiv\overline{c_p}\pmod{a}$ with $0\le\overline{c_p}<a$.  Then, we have 
\begin{enumerate}  
\item[\rm (i)] $\displaystyle m_1^{(p)}=
\begin{cases}
c_p+1&\text{if $c_p\equiv 0\pmod{a}$}\\
c_p-\overline{c_p}+a+1&\text{otherwise}. 
\end{cases}$ 
\item[\rm (ii)] $m_{a-1}^{(p)}=c_p-\overline{c_p}+a-1$. 
\end{enumerate}   
\label{lem:m-con}
\end{Lem}
\begin{proof}
As $a\nmid g_p(A)$, we see that $c_p\not\equiv 1\pmod{a}$. 
Let $c_p\equiv 0\pmod{a}$. As in \cite[Lemma 13]{GHKR17},  
$a h+1\not\in S_p$ and $a h+a-1\not\in S_p$ for $h<c_p/a$. Hence, $m_1^{(p)}=a(c_p/a)+1=c_p+1$ and $m_{a-1}^{(p)}=a(c_p/a)+a-1=c_p+a-1$. 

Let $c_p\not\equiv 0\pmod{a}$. As in \cite[Lemma 13]{GHKR17},   
$a h+1\not\in S_p$ and $a h+a-1\not\in S_p$ for $h<(c_p-\overline{c_p})/a$. Hence, $m_{a-1}^{(p)}=a\bigl((c_p-\overline{c_p})/a+1\bigr)+1=c_p-\overline{c_p}+a+1$ and $m_{a-1}^{(p)}=a\bigl((c_p-\overline{c_p})/a\bigr)+a-1=c_p-\overline{c_p}+a-1$. 
\end{proof}

For a nonnegative integer $p$ and every $i\in\{0,1,\dots\}$, there is a positive integer $k_i^{(p)}$ such that $m_i^{(p)}=k_i^{(p)}a+i$. Then $(k_0^{(p)}, k_1^{(p)},\dots,k_{a-1}^{(p)})$ are called {\it $p$-Kunz coordinates} of $S_p$. 
We can translate Lemma \ref{lem:m-con} to the language of Kunz coordinates \cite{GHKR17,HK1}.

\begin{Cor}  
Let $S_p(A)$ be an Arf numerical semigroup with $a=\min(A)$, $p$-conductor $c_p$ and $p$-Kunz coordinates $(k_0^{(p)}, k_1^{(p)},\dots,k_{a-1}^{(p)})$. Then, 
$$
k_1^{(p)}=\cl{\frac{c_p}{a}}\quad\hbox{and}\quad k_{a-1}^{(p)}=\fl{\frac{c_p}{a}}\,. 
$$ 
\label{lem:p-kunz}
\end{Cor}
\begin{proof}
When $c_p\equiv 0\pmod{a}$, by Lemma \ref{lem:m-con}, we have $m_1^{(p)}=k_1^{(p)}a+1=c_p+1$ and $m_{a-1}^{(p)}=k_{a-1}^{(p)}a+a-1=c_p+a-1$. Hence, $k_1^{(p)}=k_{a-1}^{(p)}=c_p/a$. 
When $c_p\not\equiv 0\pmod{a}$, by Lemma \ref{lem:m-con}, we have $m_1^{(p)}=k_1^{(p)}a+1=c_p-\overline{c_p}+a+1$ and $m_{a-1}^{(p)}=k_{a-1}^{(p)}a+a-1=c_p-\overline{c_p}+a-1$. Hence, $k_1^{(p)}=(c_p-\overline{c_p})/a+1$ and $k_{a-1}^{(p)}=(c_p-\overline{c_p})/a$. 
\end{proof}

\section{Final comments}  

Classically, there are many concepts among numerical semigroups. There are various generalization possibilities for the classical numerical semigroup, and even the $p$-generalization in this paper has a large amount of fluctuation. For example, the concept of {\it type} has very important roles in the symmetric properties, but there is still discussion about how to $p$-generalize its properties. Nari \cite{Nari} showed that any numerical semigroup satisfying $2 n_0(A)=g_0(A)+t_0(A)$ is almost symmetric, but it has been unknown what its $p$-generalized formula is.


\section{Appendix A (by Pedro A. Garc\'ia-S\'anchez)}  

As for taking a list of non-minimal generators, Professor Pedro A. Garc\'ia-S\'anchez suggested me to modify the function {\tt DenumerantIdeal} as follows.   

{\small \texttt{   
\begin{flushleft} 
denumerantideal:=function(l,p)\\
\quad    local msg, m, maxgen, factorizations, n, i, f, facts, toadd, gaps, ap, di,s;\\
\quad \\
\quad    if p<0 then\\
\qquad        Error("The integer argument must be non-negative");\\
\quad    fi;\\
\quad    if p=0 then\\
\qquad        return 0+s;\\
\quad    fi;\\
\quad    msg:=l;\\
\quad    m:=Minimum(l);\\
\quad    maxgen:=Maximum(msg);\\
\quad    s:=NumericalSemigroup(l);\\
\quad    factorizations:=\lbrack \rbrack;\\
\quad    gaps:=\lbrack 0\rbrack;\\
\quad    ap:=List(\lbrack 1..m\rbrack,\_->0);\\
\quad    n:=0;\\
\quad    while ForAny(ap,x->x=0) do\\
\qquad        if n>maxgen then\\
\qquad\quad            factorizations:=factorizations\{\lbrack 2..maxgen+1\rbrack\};\\
\qquad        fi;\\
\qquad        factorizations\lbrack Minimum(n,maxgen)+1\rbrack:=\lbrack\rbrack;\\
\quad \\ 
\qquad        for i in \lbrack 1..Length(msg)\rbrack do\\
\qquad            if n-msg\lbrack i\rbrack >= 0 then\\
\quad\qquad                facts:=\lbrack List(msg,x->0)\rbrack;\\
\quad\qquad                if n-msg\lbrack i\rbrack>0 then\\
\qquad\qquad                    facts:=factorizations\lbrack Minimum(n,maxgen)+1-msg\lbrack i\rbrack\rbrack;\\
\quad\qquad                fi;\\
\quad \\ 
\qquad\qquad                for f in facts do\\
\qquad\quad\quad                    toadd:=List(f);\\
\qquad\qquad\quad                    toadd\lbrack i\rbrack:=toadd\lbrack i\rbrack+1;\\
\qquad\qquad\quad                    Add(factorizations\lbrack Minimum(n,maxgen)+1\rbrack,toadd);\\
\qquad\qquad                od;\\
\qquad\quad            fi;\\
\qquad        od;\\
\quad \\ 
\qquad        factorizations\lbrack Minimum(n,maxgen)+1\rbrack:=Set(factorizations\lbrack Minimum(n,maxgen)+1\rbrack);\\
\quad \\
\qquad        if Length(factorizations\lbrack Minimum(n,maxgen)+1\rbrack)<=p then\\
\qquad\quad            Add(gaps,n);\\
\qquad        else\\
\qquad\quad           if ap\lbrack(n mod m) +1\rbrack=0 then\\
\qquad\qquad                ap\lbrack(n mod m)+1\rbrack:=n;\\
\qquad\quad            fi;\\
\qquad        fi;\\
\qquad        n:=n+1;\\
\quad   od;\\
\quad \\ 
\quad    di:=ap+s;\\
\quad    Setter(SmallElements)(di,Difference(\lbrack 0..Maximum(gaps)+1\rbrack,gaps));\\
\quad \\ 
\quad    return di;\\
end; \\
\end{flushleft}
}} 

Then we can use it in the following way: 

{\small \texttt{   
\begin{flushleft}
gap> i:=denumerantideal(\lbrack4,7,8\rbrack,2);\\
<Ideal of numerical semigroup>\\
gap> FrobeniusNumber(i);\\
33
\end{flushleft}
}} 

\end{document}